\documentclass{amsart}

\usepackage{amsfonts}
\usepackage{amssymb}

\numberwithin{equation}{section}

\theoremstyle{plain}

\newtheorem{theorem}[subsection]{Theorem}
\newtheorem{proposition}[subsection]{Proposition}
\newtheorem{lemma}[subsection]{Lemma}

\theoremstyle{definition}

\renewcommand{\leq}{\leqslant}
\renewcommand{\geq}{\geqslant}

\newcommand{\wh}{\widehat}
\newcommand{\E}{\mathbb{E}}
\newcommand{\Z}{\mathbb{Z}}
\newcommand{\R}{\mathbb{R}}
\newcommand{\C}{\mathbb{C}}

\begin{document}

\title[Appendix]{Appendix to `Roth's theorem on progressions revisited' by J. Bourgain}

\author{Tom Sanders}
\address{Department of Pure Mathematics and Mathematical Statistics\\
University of Cambridge\\
Wilberforce Road\\
Cambridge CB3 0WA\\
England } \email{t.sanders@dpmms.cam.ac.uk}

\maketitle

\section{Introduction}

In this appendix we flesh out the details of two results stated in
the paper. The first is a refinement of Fre{\u\i}man's theorem, a result which has an extensive literature (see, for
example, \cite{YB,MCC,GAF,IZR,MBN} and \cite{TCTVHV}); familiarity with \cite{MCC}
and \cite{IZR} will be assumed, although it is not logically
necessary.
\begin{theorem}\label{thm.improvedfreiman}
Suppose that $A \subset \Z$ is a finite set with $|A+A| \leq K|A|$.
Then $A$ is contained in a multidimensional arithmetic progression
$P$ with
\begin{equation*}
\dim P = O(K^{7/4}\log^3K) \textrm{ and } |P| \leq \exp(O(K^{7/4}\log^3 K))|A|.
\end{equation*}
\end{theorem}
The previous best estimates are due to Chang \cite{MCC} who showed the above result (up to logarithmic factors) with $2$ in place of $7/4$. Note that one cannot hope to improve the dimension bound past $\lfloor K-1\rfloor$, or the exponent of $K$ in the size bound below $1$; at the end of \cite{MCC} Chang (using arguments of
Bilu \cite{YB}) actually shows how to bootstrap the dimension bound
to $\lfloor K-1\rfloor$ for a small cost in the size bound. See the notes
\cite{BJGNOPACG} of Green for an exposition of this argument.

The second result we shall show is an improvement of a theorem of Konyagin and
{\L}aba from \cite{SVKIL}. For $\alpha \in
\R$ and $A \subset \R$ we write $\alpha.A:=\{\alpha a: a \in
A\}$.
\begin{theorem}\label{thm.konyaginlaba}
Suppose that $A \subset \R$ is a finite set and $\alpha \in \R$ is transcendental. Then
\begin{equation*}
 |A + \alpha.A| \gg \frac{(\log |A|)^{4/3}}{(\log \log |A|)^{8/3}} |A|.
\end{equation*}
\end{theorem}
In \cite{SVKIL} the above result was shown with a $1$ in place of $4/3$ -- again, up to factors of $\log \log |A|$ -- and it was observed that for any transcendental $\alpha$ and positive integer $N$ there is a simple construction of a set $A$ with $|A|=N$ and $|A+\alpha.A| = \exp(O(\sqrt{\log |A|}))|A|$.

The improvements in both Theorem \ref{thm.improvedfreiman} and
Theorem \ref{thm.konyaginlaba} stem from the following result, the
proof of which is the content of this appendix. It will be
appropriate for us to consider Bohr sets in $\Z/N\Z$
rather than the generalizations presented in the paper: if $\Gamma$ is a set of characters on $\Z/N\Z$ and $\delta \in (0,1]$, then we write
\begin{equation*}
B(\Gamma,\delta):=\{x \in \Z/N\Z: \|\gamma(x)\| \leq \delta \textrm{ for all } \gamma \in \Gamma\}.
\end{equation*}
\begin{theorem}\label{thm.mainrefinement}
Suppose that $A,B \subset \Z/N\Z$ have $|A+B| \leq K|B|$ and $A$ has density
$\alpha$. Then $(A-A)+(B-B)$ contains $B(\Gamma,\delta) \cap
B(\Lambda,\epsilon)$, where
\begin{equation*}
|\Gamma| =O(K^{1/2}\log \alpha^{-1}) \textrm{ and } \log \delta^{-1}
=O(K^{1/4}\log \alpha^{-1}\log (K\alpha^{-1}))
\end{equation*}
and
\begin{equation*}
|\Lambda| = O(K^{3/4}\log \alpha^{-1}) \textrm{ and } \log
\epsilon^{-1} = O(\log (K\log \alpha^{-1})).
\end{equation*}
\end{theorem}
The appendix now splits into five sections. In \S\ref{sec.bohrsetsanddissociativity} we recall
the basic facts about Bohr sets and dissociativity which we shall need; in
\S\ref{sec.densityincrement} we detail the key new density increment developed in the paper,
before completing the proof of Theorem \ref{thm.mainrefinement} in
\S\ref{sec.proofofmaintheorem}. Finally, in \S\ref{sec.freimanimprovement} we prove Theorem \ref{thm.improvedfreiman} and in \S\ref{sec.konyaginlaba} we prove Theorem \ref{thm.konyaginlaba}.

\section{Bohr sets and dissociativity}\label{sec.bohrsetsanddissociativity}

We say that a Bohr set
$B(\Gamma,\delta)$ is \emph{regular} if
\begin{equation*}
1-2^4|\Gamma||\eta| \leq
\frac{|B(\Gamma,(1+\eta)\delta)|}{|B(\Gamma,\delta)|}\leq 1+
2^4|\Gamma||\eta|\textrm{ whenever }|\Gamma||\eta| \leq 2^{-4}.
\end{equation*}
Typically Bohr sets are regular, a fact implicit in the proof of the following proposition, which may be found, for example, as Lemma 4.25 fo \cite{TCTVHV}.
\begin{proposition}\label{prop.regularvalue}
Suppose that $B(\Gamma,\delta)$ is a Bohr set. Then there is a
$\delta'$ with $\delta \leq \delta' < 2 \delta$ such that
$B(\Gamma,\delta')$ is regular.
\end{proposition}
There is a natural candidate for `approximate Haar
measure' on $B(\Gamma,\delta)$: we write $\beta_{\Gamma,\delta}$ for
the unique uniform probability measure on $B(\Gamma,\delta)$. Having identified such a measure there are various possible formulations of the `approximate annihilator' of a Bohr set and the following lemma helps us pass between them.
\begin{lemma}\label{lem.nestingofdual}
Suppose that $B(\Gamma,\delta)$ is a regular Bohr set and $\kappa>0$
is a parameter. Then
\begin{equation*}
\{\gamma:|\wh{\beta_{\Gamma,\delta}}(\gamma)| \geq \kappa\} \subset
\{\gamma:|1-\gamma(x)| \leq 2^5|\Gamma|\kappa^{-1}\delta'\delta^{-1}
\textrm{ for all } x \in B(\Gamma,\delta')\}.
\end{equation*}
\end{lemma}

Dissociativity is an important concept and for us and we shall require a
local analogue, but first we need some notation. If $\Lambda$ is a
set of characters and $m:\Lambda \rightarrow \{-1,0,1\}$ then we
write
\begin{equation*}
  m.\Lambda:=\sum_{\lambda \in \Lambda}{m_\lambda.\lambda} \textrm{
  and } \langle \Lambda \rangle :=\{m.\Lambda: m \in
  \{-1,0,1\}^\Lambda\}.
\end{equation*}
If $S$ is a symmetric neighborhood of the trivial character then we say that
a set of characters $\Lambda$ is \emph{$S$-dissociated} if
\begin{equation*}
m \in \{-1,0,1\}^\Lambda \textrm{ and } m.\Lambda \in S \textrm{
implies that } m \equiv 0.
\end{equation*}
The usual notion of dissociativity corresponds to taking
$S$ equal to the set containing just the trivial character, and typically for us $S$ will be a set of the form
$\{\gamma:|\wh{\beta_{\Gamma,\delta}}(\gamma)| \geq \kappa\}$; the following
lemma is the tool by which we make use of this notion.
\begin{lemma}\label{lem.containedinintersection}
Suppose that $B(\Gamma,\delta)$ is a regular Bohr set, $\mathcal{L}$
is a set of characters and $\Lambda$ is a maximal
$S:=\{\gamma:|\wh{\beta_{\Gamma,\delta}}(\gamma)| \geq
\kappa\}$-dissociated subset of characters. Then $\mathcal{L}$ is
contained in
\begin{equation*}
 \{\gamma:|1-\gamma(x)| \leq 2^5|\Gamma|\kappa^{-1}\delta'\delta^{-1} + 2^4|\Lambda|\epsilon \textrm{ for all } x \in
  B(\Gamma,\delta')\cap B(\Lambda,\epsilon)\}.
\end{equation*}
\end{lemma}
\begin{proof}
We begin by proving that $\mathcal{L} \subset \langle \Lambda
\rangle +S$. Suppose (for a contradiction) that there is some
character $\gamma \in \mathcal{L} \setminus (\langle \Lambda \rangle
+S)$. Let $\Lambda':=\Lambda \cup \{\gamma\}$ which is a strict
superset of $\Lambda$. We shall show that $\Lambda'$ is dissociated
contradicting the maximality of $\Lambda$. Suppose that $m \in
\{-1,0,1\}^{\Lambda'}$ is such that $m.\Lambda' \in S$. We have
three cases.
\begin{enumerate}
  \item $m_{\gamma}=0$ in which case $m|_{\Lambda}.\Lambda \in S$
  and so $m|_{\Lambda}\equiv 0$ by $S$-dissociativity of $\Lambda$. It
  follows that $m \equiv 0$. \item $m_{\gamma}=1$ in which case $\gamma \in - m|_{\Lambda}.\Lambda +S\subset \langle \Lambda \rangle +S$ which contradicts the fact that $\gamma \in \mathcal{L}
\setminus (\langle \Lambda \rangle +S)$. \item $m_{\gamma}=-1$ in
which case $\gamma \in m|_{\Lambda}.\Lambda +S\subset \langle
\Lambda
  \rangle +S$ which contradicts the fact that $\gamma \in \mathcal{L}
\setminus (\langle \Lambda \rangle +S)$.
\end{enumerate}
Thus $m \equiv 0$ and $\Lambda'$ is $S$-dissociated as claimed. This
contradiction proves that $\mathcal{L} \subset \langle \Lambda
\rangle +S$.

By Lemma \ref{lem.nestingofdual} we have
\begin{equation*}
S \subset \{\gamma: |1-\gamma(x)| \leq
2^5|\Gamma|\kappa^{-1}\delta'\delta^{-1} \textrm{ for all } x \in
B(\Gamma,\delta')\}.
\end{equation*}
Now if $\gamma \in \langle \Lambda \rangle$ then $\gamma =
\sum_{\lambda \in \Lambda}{m_\lambda.\lambda}$ so
\begin{eqnarray*}
|1-\gamma(x)| \leq \sum_{\lambda \in \Lambda}{|1-\lambda(x)|}
 &= &\sum_{\lambda \in
 \Lambda}{\sqrt{2(1-\cos(4\pi\|\lambda(x)\|))}}\\ &
\leq & \sum_{\lambda \in \Lambda}{4\pi\|\lambda(x)\|} \leq
2^4|\Lambda| \sup_{\lambda \in \Lambda}{\|\lambda(x)\|}.
\end{eqnarray*}
It follows that
\begin{equation*}
\langle \Lambda \rangle \subset \{\gamma: |1-\gamma(x)|\leq
2^4|\Lambda|\epsilon \textrm{ for all } x \in B(\Lambda,\epsilon)\}.
\end{equation*}
The result follows from the triangle inequality.
\end{proof}

\section{The density increment}\label{sec.densityincrement}

The objective of this section is to prove the following proposition
which reflects the main innovation of the paper. The key idea
of the proposition is that if we have a large number of highly
independent characters at which $\wh{1_A}$ is large then they
induce `orthogonal' density increments which can consequently be
done simultaneously resulting in a more favourable width
reduction to density increment ratio in our Bohr sets.
\begin{proposition}\label{prop.densityincrement}
Suppose that $B(\Gamma,\delta)$ is a regular Bohr set, $A \subset
B(\Gamma,\delta)$ has relative density $\alpha$ and $\rho \in (0,1]$
is a parameter. Suppose, further, that there is a Bohr set
$B(\Gamma,\delta')$ with
\begin{equation*}
2\left(\frac{\rho\alpha }{2(1+|\Gamma|)}\right)^{2^6}\delta  \leq
\delta' \leq 4\left(\frac{\rho\alpha
}{2(1+|\Gamma|)}\right)^{2^6}\delta,
\end{equation*}
and a $\{\gamma:|\wh{\beta_{\Gamma,\delta'}}(\gamma)| \geq
1/3\}$-dissociated set $\Lambda$ of at least $2^7\rho^{-1} (1+\log
\alpha^{-1})$ characters such that for each $\lambda \in \Lambda$ we
have $|\wh{1_A d\beta_{\Gamma,\delta}}(\lambda)| \geq \rho \alpha$.
Then there is a regular Bohr set $B(\Gamma',\delta'')$ with
\begin{equation*}
\delta'' \geq \left(\frac{\rho\alpha
}{2(1+|\Gamma|)}\right)^{2^6}\delta \textrm{ and } |\Gamma'| \leq |\Gamma| + \rho |\Lambda|/2^4(1+\log \alpha^{-1})
\end{equation*}
such that
\begin{equation*}
\|1_A \ast \beta_{\Gamma',\delta''}\|_\infty \geq
\alpha\left(1+\frac{\rho^2 |\Lambda|}{2^{12}(1+\log
\alpha^{-1})}\right).
\end{equation*}
\end{proposition}
This proposition is proved using Proposition (*) which essentially introduces Riesz products. We
shall now formalize some appropriate notation and definitions to
deal with them. Suppose that $\Lambda$ is a symmetric set of
characters. $\omega:\Lambda \rightarrow D:=\{z \in \C:|z| \leq 1\}$
is \emph{hermitian} if $\omega(\lambda) =
\overline{\omega(-\lambda)}$ for all $\lambda \in \Lambda$. Given a
hermitian $\omega:\Lambda \rightarrow D$ we define the product
\begin{equation*}
p_{\omega}:=\prod_{\{\lambda,-\lambda\} \subset
\Lambda}{\left(1+\frac{\omega(\lambda)\lambda +
\overline{\omega(\lambda)\lambda}}{2}\right)}
\end{equation*}
and call it a \emph{Riesz product}.

To pass between the notion of dissociaitivity defined in the
previous section and the `Riesz product condition' towards the end
of Proposition (*) we use the following technical lemma.
\begin{lemma}\label{lem.rieszproductestimate}
Suppose that $B(\Gamma,\delta)$ is a regular Bohr set, $\Lambda$ is
a symmetric $\{\gamma:|\wh{\beta_{\Gamma,\delta'}}(\gamma)| \geq
1/3\}$-dissociated set of characters and $\omega:\Lambda \rightarrow
D$ is hermitian. Then
\begin{equation*}
\int{p_\omega d\beta_{\Gamma,\delta+\delta''}} \leq
1+2^7(|\Gamma||\Lambda|\delta'\delta^{-1})^{1/2}
\end{equation*}
for all $\delta'' \leq \delta'$.
\end{lemma}
\begin{proof}
We need to introduce some smoothed measures. Let $L$ be an integer
to be optimized later and write $\tilde{\beta}$ for the measure
$\beta_{\Gamma,\delta+\delta''+2L\delta'}\ast \beta_{\Gamma,\delta'}
\ast ... \ast \beta_{\Gamma,\delta'}$ where $\beta_{\Gamma,\delta'}$
occurs $2L$ times. Now $p_{\omega} \geq 0$ and $\tilde{\beta}$ is
uniform on $B(\Gamma,\delta+\delta'')$ so
\begin{eqnarray*}
\int{p_\omega d\beta_{\Gamma,\delta+\delta''}} & \leq &
\int{p_{\omega}d\tilde{\beta}} \times \frac{|B(\Gamma,\delta+\delta''+2L\delta')|}{|B(\Gamma,\delta+\delta'')|}\\
& \leq &
(1+2^8|\Gamma|L\delta'\delta^{-1})\int{p_{\omega}d\tilde{\beta}};
\end{eqnarray*}
the last inequality by regularity and the fact that $\delta'' \leq
\delta'$. Now Plancherel's theorem tells us that
\begin{equation*}
\int{p_\omega d\tilde{\beta}} \leq \sum_{\gamma \in \langle \Lambda
\rangle}{|\wh{\beta_{\Gamma,\delta'}}(\gamma)|^{2L}} \leq 1+3^{-L}
\leq 1+L^{-1}
\end{equation*}
if $L \geq |\Lambda|$. $L$ can now be optimized with ease.
\end{proof}

The content of the next proof is simply the observation that a Riesz
product on $\Lambda$ is roughly constant on a small enough Bohr set
on the characters $\Lambda$.

\begin{proof}[Proof of Proposition \ref{prop.densityincrement}]
For each $\lambda \in \Lambda$ let $\omega(\lambda)$ be a complex
number such that
\begin{equation*}
\omega(\lambda) \wh{1_Ad\beta_{\Gamma,\delta}}(\lambda) =
|\wh{1_Ad\beta_{\Gamma,\delta}}(\lambda)|.
\end{equation*}
Note that $\omega$ is hermitian since $1_Ad\beta_{\Gamma,\delta}$ is
real. We let $\Phi$ be the set $\{(\omega(\lambda)\lambda +
\overline{\omega(\lambda)\lambda})/2:\lambda \in \Lambda\}$ so that
$|\Lambda| \geq |\Phi| \geq |\Lambda|/2$. From the definition of
$\omega$ we see that $\langle 1_A , \phi \rangle \geq \rho\alpha$
for all $\phi \in \Phi$ and since
\begin{equation*}
\delta' \leq \delta/2^{14}(1+|\Lambda|)(1+|\Gamma|),
\end{equation*}
Lemma \ref{lem.rieszproductestimate} applies and we have
\begin{equation*}
\int{\prod_{\phi \in \Phi}{(1+a(\phi)\phi)}d\beta_{\Gamma,\delta}}
\leq 2
\end{equation*}
for all $a:\Phi \rightarrow [-1,1]$.

We apply Proposition (*) (with constants which come out of the proof) to conclude that there is a
set $\Phi' \subset \Phi$ with
\begin{equation*}
|\Phi'| \leq \frac{\rho |\Lambda|}{2^5(1+\log \alpha^{-1})}
\end{equation*}
such that
\begin{equation*}
\int{1_A \prod_{\phi \in \Phi'}{(1+\phi)}d\beta_{\Gamma,\delta}}
\geq \alpha\left(1+\frac{\rho^2|\Lambda|}{2^8(1+\log
\alpha^{-1})}\right).
\end{equation*}
Let $\Lambda'$ be the subset of $\Lambda$ such that
$\Phi'=\{(\omega(\lambda)\lambda +
\overline{\omega(\lambda)\lambda})/2:\lambda \in \Lambda'\}$ and
$\omega':=\omega|_{\Lambda'}$ so that $p_{\omega'} = \prod_{\phi \in
\Phi'}{(1+\phi)}$. We take $\Gamma':=\Gamma \cup \Lambda'$ and it
follows that
\begin{equation*}
|\Gamma'| \leq |\Gamma| + \frac{\rho |\Lambda|}{2^4(1+\log
\alpha^{-1})}.
\end{equation*}
Now place some total order $<$ on $\Phi'$. Then
\begin{equation*}
p_{\omega'}(x+y)-p_{\omega'}(x)=\sum_{\phi' \in \Phi'}{\prod_{\phi <
\phi'}{(1+\phi(x+y))}\left(\phi'(x+y)-\phi'(x)\right)\prod_{\phi' <
\phi}{(1+\phi(x))}}.
\end{equation*}
It follows that
\begin{equation*}
|p_{\omega'}(x+y)-p_{\omega'}(x)| \leq \sum_{\phi' \in
\Phi'}{\prod_{\phi < \phi'}{(1+\phi(x+y))}|\lambda(y)-1|\prod_{\phi'
< \phi}{(1+\phi(x))}},
\end{equation*}
and hence if $y \in B(\Gamma',\delta'')$ we conclude that
\begin{equation*}
|p_{\omega'}(x+y)-p_{\omega'}(x)| \leq 2^4\delta''\sum_{\phi' \in
\Phi'}{\prod_{\phi < \phi'}{(1+\phi(x+y))}\prod_{\phi' <
\phi}{(1+\phi(x))}}.
\end{equation*}
If we define
\begin{equation*}
\omega_{\phi'}(\lambda):=\begin{cases} \omega(\lambda)\lambda(y) &
\textrm{ if } (\omega(\lambda)\lambda +
\overline{\omega(\lambda)\lambda})/2 < \phi'\\
0 & \textrm{ if } (\omega(\lambda)\lambda +
\overline{\omega(\lambda)\lambda})/2 =
\phi'\\
 \omega(\lambda) &
\textrm{ if } (\omega(\lambda)\lambda +
\overline{\omega(\lambda)\lambda})/2 > \phi'
\end{cases}
\end{equation*}
then this last expression can be written as
\begin{equation*}
|p_{\omega'}(x+y)-p_{\omega'}(x)| \leq 2^4\delta''\sum_{\phi' \in
\Phi'}{p_{\omega_{\phi'}}(x)}.
\end{equation*}
Hence, by Lemma \ref{lem.rieszproductestimate}, we have
\begin{equation*}
\int{p_{\omega_{\phi'}}d\beta_{\Gamma,\delta}} \leq 2,
\end{equation*}
whence
\begin{equation*}
\int{1_A|p_{\omega'} \ast \beta_{\Gamma',\delta''} -
p_{\omega'}|d\beta_{\Gamma,\delta}} \leq 2^5\delta''|\Phi'|.
\end{equation*}
Pick $\delta''$ satisfying the lower bound of the proposition and
regular for $\Gamma'$ by Proposition \ref{prop.regularvalue}, such that
\begin{equation*}
|\int{1_A p_{\omega'} \ast
\beta_{\Gamma',\delta''}d\beta_{\Gamma,\delta}}-\int{1_A p_{\omega'}
d\beta_{\Gamma,\delta}}| \leq 2^5|\Lambda|\delta'' \leq \frac{\alpha
\rho^2}{2^{10}(1+\log \alpha^{-1})},
\end{equation*}
where the last inequality is by choice of $\delta''$. Consequently
\begin{equation*}
\int{1_A p_{\omega'} \ast
\beta_{\Gamma',\delta''}d\beta_{\Gamma,\delta}} \geq
\alpha\left(1+\frac{\rho^2|\Lambda|}{2^9(1+\log
\alpha^{-1})}\right).
\end{equation*}
But
\begin{eqnarray*}
\int{1_A p_{\omega'} \ast
\beta_{\Gamma',\delta''}d\beta_{\Gamma,\delta}} & \leq &
\frac{|B(\Gamma,\delta+\delta'')|}{|B(\Gamma,\delta)|}
\times \int{1_A \ast \beta_{\Gamma',\delta''} p_{\omega'}
d\beta_{\Gamma,\delta+\delta''}}\\ & \leq &
(1+2^4|\Gamma|\delta''\delta^{-1})\|1_A \ast
\beta_{\Gamma',\delta''}\|_\infty\\ & & \times
(1+2^7(|\Gamma||\Lambda'|\delta'\delta^{-1})^{1/2}).
\end{eqnarray*}
It follows that
\begin{equation*}
\|1_A \ast \beta_{\Gamma',\delta''}\|_\infty \geq
(1-2^{10}(|\Gamma||\Lambda|\delta'\delta^{-1})^{1/2})\int{1_A
p_{\omega'} \ast \beta_{\Gamma',\delta''}d\beta_{\Gamma,\delta}},
\end{equation*}
from which we retrieve the result.
\end{proof}

\section{The proof of Theorem \ref{thm.mainrefinement}}\label{sec.proofofmaintheorem}

The proof is iterative with the following lemma as the driving ingredient.

\begin{lemma}[Iteration lemma]\label{lem.iterationlemma}
Suppose that $B(\Gamma,\delta)$ is a regular Bohr set and $A \subset
\Z/N\Z$ has relative density $\alpha$ in $B(\Gamma,\delta)$. Suppose,
further, that there is a set $B \subset \Z/N\Z$ such that $|A + B| \leq K |B|$. Then at least one of
the following conclusions is true.
\begin{enumerate}
\item \label{item.case1} $(A-A)+(B-B)$ contains $B(\Gamma,\delta'')\cap
B(\Lambda,\epsilon)$ where $\Lambda$ is a set of size $O(K^{3/4}\log
\alpha^{-1})$ and
\begin{equation*}
\delta''^{-1}= (2\alpha^{-1} K(1+|\Gamma|))^{O(1)}\delta^{-1}
\textrm{ and } \epsilon^{-1} = O(K\log \alpha^{-1}).
\end{equation*}
\item \label{item.case2} There is a regular Bohr set
$B(\Gamma',\delta'')$ with
\begin{equation*}
|\Gamma'| \leq |\Gamma| + O(K^{1/4}) \textrm{ and } \delta''^{-1} =
(2\alpha^{-1}K(1+|\Gamma|))^{O(1)}\delta^{-1},
\end{equation*}
and such that
\begin{equation*}
\|1_A \ast \beta_{\Gamma',\delta''}\|_{\infty} \geq
\alpha(1+2^{-7}K^{-1/4}).
\end{equation*}
\end{enumerate}
\end{lemma}
\begin{proof}
Apply Proposition \ref{prop.regularvalue} to pick $\delta'$ regular for $\Gamma$
such that
\begin{equation*}
2\left(\frac{\alpha}{4\sqrt{K}(1+|\Gamma|)}\right)^{2^6}\delta \leq
\delta' \leq 4\left(\frac{
\alpha}{4\sqrt{K}(1+|\Gamma|)}\right)^{2^6}\delta.
\end{equation*}
Let $\Lambda$ be a maximal
$\{\gamma:|\wh{\beta_{\Gamma,\delta'}}(\gamma)| \geq
1/3\}$-dissociated subset of
\begin{equation*}
\mathcal{L}:=\{\gamma :|\wh{1_{A}d\beta_{\Gamma,\delta}}(\gamma)| \geq
\alpha/2\sqrt{K}\}.
\end{equation*}
If $|\Lambda| \leq 2^8K^{3/4}(1+\log \alpha^{-1})$ then apply Lemma
\ref{lem.containedinintersection} to see that $\mathcal{L}$ is
contained in
\begin{equation*}
\{\gamma:|1-\gamma(x)| \leq 1/4 \textrm{ for all } x \in
B(\Gamma,\delta'')\cap B(\Lambda,\epsilon)\},
\end{equation*}
where
\begin{equation*}
\delta''= \frac{\delta'}{2^{10}(1+|\Gamma|)} \textrm{ and } \epsilon
= \frac{1}{2^{15}K(1+\log \alpha^{-1})}.
\end{equation*}
Write $|B|=\beta N$. By the Cauchy-Schwarz inequality we have
\begin{eqnarray*}
\beta^2\alpha^2& =&\left(\E{1_{B} \ast
(1_{A}d\beta_{\Gamma,\delta})}\right)^2\\ & \leq &
\E{(1_{B} \ast (1_{A}d\beta_{\Gamma,\delta}))^2} K\beta.
\end{eqnarray*}
It follows that if we write $f:=1_{B} \ast
(1_{A}d\beta_{\Gamma,\delta}) \ast 1_{-B} \ast
(1_{-A}d\beta_{\Gamma,\delta})$ then $f(0) \geq \alpha^2\beta/K$. By the inversion formula we have
\begin{equation*}
  f(x)=\sum_{\gamma}{|\wh{1_B}(\gamma)|^2|\wh{1_Ad\beta_{\Gamma,\delta}}(\gamma)|^2\gamma(x)},
\end{equation*}
whence, by Parseval's theorem, we have
\begin{eqnarray*}
|f(0) - f(x)| & \leq & \sum_{\gamma \in
\mathcal{L}}{|\wh{1_B}(\gamma)|^2|\wh{1_Ad\beta_{\Gamma,\delta}}(\gamma)|^2|\gamma(x)-1|}\\
& & + 2\sum_{\gamma \not\in
\mathcal{L}}{|\wh{1_B}(\gamma)|^2|\wh{1_Ad\beta_{\Gamma,\delta}}(\gamma)|^2}\\
& \leq & \sup_{\gamma \in \mathcal{L}}{|1-\gamma(x)|}\sum_{\gamma}{|\wh{1_B}(\gamma)|^2|\wh{1_Ad\beta_{\Gamma,\delta}}(\gamma)|^2}\\
& & + \beta\alpha^2/2K\\
& \leq &
   \left(\sup_{\gamma \in \mathcal{L}}{|1-\gamma(x)|} + 1/2\right) f(0)
\leq 3f(0)/4
\end{eqnarray*}
if $x \in B(\Gamma,\delta'')\cap B(\Lambda,\epsilon)$. It follows
that we are in case (\ref{item.case1}).

In the other case we discard (if necessary) just enough elements of
$\Lambda$ to ensure that the inequality $|\Lambda| \leq
2^9K^{3/4}(1+\log \alpha^{-1})$ holds and then apply Proposition
\ref{prop.densityincrement} with parameter $\rho=1/2\sqrt{K}$. It follows that
there is a regular Bohr set $B(\Gamma',\delta'')$ with
\begin{equation*}
\delta'' \geq (\alpha/2K(1+|\Gamma|))^{2^7}\delta \textrm{ and }
|\Gamma'| \leq |\Gamma| + 2^4K^{1/4},
\end{equation*}
and
\begin{equation*}
\|1_{A} \ast \beta_{\Gamma',\delta''}\|_\infty \geq
\alpha(1+2^{-7}K^{-1/4}).
\end{equation*}
It follows that we are in case (\ref{item.case2}).
\end{proof}
Iterating this to yield Theorem \ref{thm.mainrefinement} is a simple
exercise.
\begin{proof}[Proof of Theorem \ref{thm.mainrefinement}]
We construct a sequence of regular Bohr sets $B(\Gamma_k,\delta_k)$
iteratively initializing with $\Gamma_0$ as the set containing the trivial character and
$\delta_0=1$ which has $B(\Gamma_0,\delta_0)$ regular for trivial
reasons. Write $\alpha_k=\|1_A \ast
\beta_{\Gamma_k,\delta_k}\|_\infty$ so that $\alpha_0=\alpha$ and
let $x_k$ be such that $1_A \ast
\beta_{\Gamma_k,\delta_k}(x_k)=\alpha_k$. We apply Lemma \ref{lem.iterationlemma}
repeatedly to the sets $x_k-A$ and the Bohr sets
$B(\Gamma_k,\delta_k)$. If after $k$ steps of the iteration we have
never found ourselves in the first case of Lemma \ref{lem.iterationlemma} then
\begin{equation*}
  \alpha_k \geq \alpha(1+2^{-7}K^{-1/4})^k, |\Gamma_k| = O(K^{1/4}k)
  \textrm{ and } \delta_k^{-1} = (2\alpha^{-1} K^{3/4}k)^{O(k)}.
\end{equation*}
Since $\alpha_k \leq 1$ the first of these ensures that
$k=O(K^{1/4}(\log \alpha^{-1})$ and so there is some $k$ of size
$O(K^{1/4}\log \alpha^{-1})$ for which we end up in the first case
of Lemma \ref{lem.iterationlemma}, and at that stage we have
\begin{equation*}
  |\Gamma_k| = O(K^{1/2}\log \alpha^{-1})
\end{equation*}
and
\begin{equation*}
\delta_k^{-1} =
  \exp(O(K^{1/4}\log \alpha^{-1}\log\alpha^{-1}K)).
\end{equation*}
The result follows.
\end{proof}

\section{Improving Fre{\u\i}man's theorem: the proof of Theorem \ref{thm.improvedfreiman}}\label{sec.freimanimprovement}

Ruzsa's proof of Fre{\u\i}man's theorem in \cite{IZR} naturally splits into four stages:
finding a good model; Bogolio{\`u}boff's argument; determining the
structure of Bohr sets; and, Chang's pullback and covering argument.
The improvement of this work arises from replacing
Bogolio{\`u}boff's argument by the more sophisticated Theorem
\ref{thm.mainrefinement}.

To `find a good model' we use the following proposition due to Ruzsa
which essentially appears as Theorem 8.9 in \cite{MBN} for example.
\begin{proposition}\label{prop.ruzsamodel}
Suppose that $A$ is a finite set of integers with $|A+A| \leq K|A|$.
Then there is an integer $N$ with $N= K^{O(1)}|A|$ and a set $A'
\subset A$ with $|A'| \geq |A|/8$ such that $A'$ is Fre{\u\i}man
8-isomorphic to a subset of $\Z/N\Z$.
\end{proposition}
We have already dealt with `Bogolio{\`u}boff's argument', so we turn
to determining the structure of Bohr sets. It was a key insight of
Ruzsa in \cite{IZR} that Bohr sets contain large multidimensional
progressions. Fortunately the same is true for intersections of Bohr
sets.

The following is a straightforward generalization of
\cite[Proposition 4.23]{TCTVHV}. It shows that the intersection of
two Bohr sets contains a large multidimensional progression. If $\Gamma$ is a set of characters on $\Z/N\Z$ and $(\delta_\gamma)_{\gamma \in \Gamma} \in (0,1]^\Gamma$ then we write
\begin{equation*}
B(\Gamma,(\delta_\gamma)_{\gamma \in \Gamma}):=\{x \in \Z/N\Z: \|\gamma(x)\| \leq \delta_\gamma \textrm{ for all } \gamma \in \Gamma\}.
\end{equation*}
\begin{proposition}\label{prop.largeprogressionsinbohrsets}
Suppose that $\Gamma$ is a set of characters on $\Z/N\Z$ and $(\delta_\gamma)_{\gamma \in \Gamma} \in (0,1]^\Gamma$. Then $B(\Gamma,(\delta_\gamma)_{\gamma \in \Gamma})$ contains a
symmetric multidimensional progression $P$ of dimension $|\Gamma|$
and size at least $\prod_{\gamma \in
\Gamma}{(\delta_{\gamma}/|\Gamma|)} N$.
\end{proposition}
Finally, `Chang's pullback and covering argument' is the following
result which converts a large progression contained in $2A-2A$
into a progression containing $A$. It may be found, for example,
as \cite[Proposition 26]{BJGSTSA}
\begin{proposition}\label{prop.changspullback}
Suppose that $A$ is a finite set of integers with $|A+A| \leq K|A|$
and $2A-2A$ contains a multidimensional progression of
dimension $d$ and size $\eta|A|$. Then $A$ is contained in a
multidimensional progression of size at most $\exp(O(d+K \log
K\eta^{-1}))|A|$ and dimension at most $O(d+K \log K\eta^{-1})$.
\end{proposition}
With these ingredients we can now prove Theorem
\ref{thm.improvedfreiman}.
\begin{proof}[Proof of Theorem \ref{thm.improvedfreiman}] We apply Proposition
\ref{prop.ruzsamodel} to get an integer $N$ with $N=K^{O(1)}|A|$ and
a set $A'\subset A$ with $|A'| \geq |A|/8$ such that $A'$ is Fre{\u\i}man
8-isomorphic to a subset $A''$ of $\Z/N\Z$. Note that $A''$ has
$|A''+A''| \leq K|A| \leq 8K|A''|$ and density $K^{-O(1)}$. Apply
Theorem \ref{thm.mainrefinement} to get that $2A''-2A''$
contains $B(\Gamma,\delta)\cap B(\Lambda,\epsilon)$
where
\begin{equation*}
  |\Gamma| = O(K^{1/2}\log K) \textrm{ and } \log \delta^{-1} = O(K^{1/4}\log^2K),
\end{equation*}
and
\begin{equation*}
  |\Lambda| = O(K^{3/4}\log K) \textrm{ and } \log \epsilon ^{-1} =
  O(\log K).
\end{equation*}
Proposition \ref{prop.largeprogressionsinbohrsets} then ensures that
$2A''-2A''$ contains a multidimensional progression of
dimension $O(K^{3/4}\log K)$ and size at least
$\exp(-O(K^{3/4}\log^3K))N$. Since $A''$ is Fre{\u\i}man 8-isomorphic to
a subset of $A$ we have that $2A''-2A''$ is Fre{\u\i}man 2-isomorphic
to a subset of $2A-2A$ which thus contains a multidimensional
progression of dimension $O(K^{3/4}\log K)$ and size at least
$\exp(-O(K^{3/4}\log^3K))|A|$. The result follows from Proposition
\ref{prop.changspullback}.
\end{proof}

It is worth remarking that for the purpose of improving
the bounds in Fre{\u\i}man's theorem for general abelian groups (the current best such appearing in the paper \cite{BJGIZR} of Green and Ruzsa) it would be desirable to pay closer attention to the
$\alpha$-dependencies in Theorem \ref{thm.mainrefinement}. These
contribute logarithmic terms in $\mathbb{Z}/N\mathbb{Z}$, but polynomial
terms when we do not have a modelling lemma of the strength of
Proposition \ref{prop.ruzsamodel}, as is the case in general.

\section{Improving the Konyagin-{\L}aba theorem: the proof of Theorem \ref{thm.konyaginlaba}}\label{sec.konyaginlaba}

We require two preliminary results. 
\begin{lemma}\label{lem.iteratedgrowth}
Suppose that $A \subset \R$, $\alpha \in \R\setminus \{0\}$ and
$|A+\alpha.A| \leq K|A|$. Then
\begin{equation*}
|(A-A)+\alpha.(A-A) + \dots + \alpha^{l-1}.(A-A)|
\leq K^{O(l)}|A|.
\end{equation*}
\end{lemma}
\begin{proof}
Write $A':=A-A$ and $B':=\alpha.(A-A)$. By the Pl{\"u}nnecke-Ruzsa estimates
(\cite[Corollary 6.27]{TCTVHV}) we have $|kB'| \leq K^{2k}|A|$ for all
$k$. Since $\alpha \neq 0$ we have that $|kB'|=|kA'|$, whence $|kA'|
\leq K^{2k}|A|$ and, in particular, $|A'-A'| \leq K^4|A|$ and $|3A'-3A'|
\leq K^{12}|A|$. By this second inequality and Ruzsa's covering lemma
(\cite[Lemma 2.14]{TCTVHV}) there is a set $S$ with $|S| \leq K^{12}$
such that $3A'-2A' \subset A'-A' + S$. Again, by Ruzsa's covering lemma
and the fact that $|A+\alpha.A| \leq K|A|$ there is also a set $T$
with $|T|\leq K$ and $\alpha.A \subset A-A + T$, whence $B' \subset A'-A' + 2T-2T$. Put $T':=2T-2T$ and note that $|T'| \leq K^4$.

Now, write $B_l:=A' + \alpha.A' + \dots + \alpha^{l-1}.A'$ and define a
sequence of sets $T_l$ recursively by letting $T_1$ be some set
containing precisely one element of $A'$ and $T_{l+1}:=S+T'-T'+\alpha
T_l$. We shall show by induction that $B_l \subset A'-A' + T_l$. For
$l=1$ this is immediate. Assume that we have shown the result for
some $l$. Then
\begin{eqnarray*}
B_{l+1} = A'+ \alpha.B_l & \subset & A' + \alpha.A' - \alpha.A' + \alpha.T_l\\ & \subset & 3A'-2A' + T' - T' + \alpha.T_l\\ & \subset & A'-A' +
S+T'-T'+\alpha.T_l = A'-A'+T_{l+1}.
\end{eqnarray*}
The claim follows. It remains to note that $|B_l| \leq |A'-A'||T_l|
\leq K^4|A||T_l| \leq K^{20l-16}|A|$ as required.
\end{proof}
Note that a direct application of the Pl{\"u}nnecke-Ruzsa estimates gives only a bound of the form $K^{O(l^2)}|A|$.

The following is a straightforward modification of Proposition \ref{prop.ruzsamodel}.
\begin{proposition}\label{prop.ruzsamodelbooster}
Suppose that $A$ and $B$ are finite sets of integers with $|A+B|
\leq K\min\{|A|,|B|\}$ and $k \geq 2$ is a positive integer. Then
there is an integer $N$ with $N= K^{O(k)}\min\{|A|,|B|\}$, a
subset of $A$ of size at least $|A|/k$, a subset $B$, and a Fre{\u\i}man
$k$-isomorphism mapping these sets to $A''$ and $B''$ in $\Z/N\Z$ respectively. Furthermore, $|A''+B''| \leq kK\min\{|A''|,|B''|\}$.
\end{proposition}
\begin{proof}
There is clearly a Fre{\u\i}man $k$-isomorphism which maps $A$ and $B$
onto $A'$ and $B'$, respectively, some subsets of $\Z/p\Z$ for a
sufficiently large prime $p$. Since $k \geq 2$ we have $|A'+B'|
=|A+B|$ and consequently we shall assume that $A$ and $B$ are
subsets of $\Z/p\Z$.

Suppose that $q \in \Z/p\Z^*$ and define $\phi:\Z/p\Z \rightarrow
\Z/p\Z; x \mapsto qx$ and $\phi':\Z/p\Z \mapsto \Z$ to be the map
which takes $x+p\Z$ to its least non-negative member. The range of
$\phi'$ is partitioned by the $k$ sets
\begin{equation*}
I_j:=\left\{x \in \Z: \frac{j-1}{k}p \leq x < \frac{j}{k}p-1\right\}
\textrm{ for } j \in [k],
\end{equation*}
and, moreover, $\phi'|_{\phi'^{-1}(I_j)}$ is clearly a Fre{\u\i}man
$k$-isomorphism for each $j \in [k]$.

By the pigeon-hole principle there is some $j=j(q)$ such that
\begin{equation*}
A(q):=(\phi' \circ \phi)^{-1}(I_{j} \cap A \textrm{ has } |A(q)|
\geq |A|/k;
\end{equation*}
similarly there is some $x=x(q) \in \Z/p\Z$ such that
\begin{equation*}
B(q):=(\phi' \circ \phi')^{-1}(I_{j(q)}) \cap (x+B) \textrm{ has
}|B(q)| \geq |B|/k.
\end{equation*}
Put $C(q)=A(q) \cup B(q)$, and note that $\phi' \circ \phi|_{C(q)}$
is a Fre{\u\i}man $k$-isomorphism.

Finally, let $N:= |k(A+B)-k(A+B)|$ and $\phi'':\Z \rightarrow
\Z/N\Z$ be the usual quotient map. For every $q$, $\psi:=\phi''
\circ \phi' \circ \phi|_{C(q)}$ is a Fre{\u\i}man $k$-homomorphism;
it is also a Fre{\u\i}man $k$-isomorphism: put
\begin{equation*}
I(q):=\left\{q(\sum_{i=1}^k{a_i} -
\sum_{i=1}^k{a_i'}):a_1,\dots,a_k,a_1',\dots,a_k' \in C(q)\right\},
\end{equation*}
and
\begin{equation*}
M(N,p):=\{kN+p\Z: 0 \leq k \leq \lfloor p/N\rfloor\}.
\end{equation*}
$\psi$ is a Fre{\u\i}man $k$-isomorphism for some $q$ if $I(q) \cap
M(N,p)=\{p\Z\}$. For each $q$ there are at most $|kC(q) - kC(q)|-1
\leq N-1$ non-zero elements in $I(q)$, whence if $q$ is chosen
uniformly at random from $\Z/p\Z^*$, we have
\begin{eqnarray*}
\E_q{|I(q) \cap M(N,p)|} & = & 1+\frac{(|M(N,p)|-1)(|I(q)|-1)}{p-1}\\
& \leq & 1+ \frac{(\lfloor p/N\rfloor -1)}{p-1}.(N-1)<2.
\end{eqnarray*}
We conclude that there is some $q$ such that $|I(q) \cap M(N,p)|<2$,
and so $I(q) \cap M(N,p)=\{p\Z\}$, and $\psi$ is a Fre{\u\i}man
$k$-isomorphism.

To finish the proof we put $A'':=\psi(A(q))$ and $B'':=\psi(B(q))$,
and since $\psi$ is a Fre{\u\i}man 2-isomorphism,
\begin{eqnarray*}
|A'' + B''| = |A(q)+B(q)| & \leq & |A+(x(q)+B)|\\ & \leq &
K\min\{|A|,|B|\} \leq kK\min\{|A''|,|B''|\}.
\end{eqnarray*}
The bound on $N$ follows from the Pl{\"u}nnecke-Ruzsa estimates
(\cite[Corollary 6.27]{TCTVHV}).
\end{proof}
The following argument is due to J. Bourgain.
\begin{proof}[Proof of Theorem \ref{thm.konyaginlaba}]
Since $A\cup \alpha.A$ is finite it generates a finite dimensional $\Z$-module in $\R$, whence $A \cup \alpha.A$ is Fre{\u\i}man isomorphic of all orders to a subset $A' \cup B'$ of $\Z^d$. (Here, of course, $A$ is isomorphic to $A'$ and $\alpha.A$ to $B'$.)

Let $L$ be the the largest coefficient occurring
in any element of $A' \cup B'$ (we may assume, by a translation, that they are all positive). The map
\begin{equation*}
(x_1,\dots,x_d) \mapsto x_1 + 17Lx_2+(17L)^2x_3+\dots+(17L)^{d-1}x_d
\end{equation*}
is a Fre{\u\i}man $8$-isomorphism of $A' \cup B'$ into $\Z$ from
which it follows that $|A'+B'| \leq K\min\{|A'|,|B'|\}$ since $|A'|
= |B'|$. By Proposition \ref{prop.ruzsamodelbooster} there is an integer $N$ with $N=K^{O(1)}|A|$, a subset of $A'$ of size at least $|A|/8$, a subset of $B'$ and a Fre{\u\i}man $8$-isomorphism mapping these sets to $A''$ and $B''$ in $\Z/N\Z$ respectively such that $|A'' +B''| \leq
8K\min\{|A''|,|B''|\}$.

It follows from Theorem \ref{thm.mainrefinement} that $(A''-A'')+(B''-B'')$
contains $B(\Gamma,\delta) \cap B(\Lambda,\epsilon)$ where
\begin{equation*}
|\Gamma| =O(K^{1/2}\log K) \textrm{ and } \log \delta^{-1}
=O(K^{1/4}\log K)
\end{equation*}
and
\begin{equation*}
|\Lambda| = O(K^{3/4}\log K) \textrm{ and } \log \epsilon^{-1} =
O(\log K).
\end{equation*}
(Essentially) by Lemma 2.0(i) of the paper we have that
\begin{equation*}
|B(\Gamma,\eta\delta) \cap B(\Lambda,\eta\epsilon)| \geq
(\eta\delta)^{|\Gamma|}(\eta\epsilon)^{|\Lambda|}N.
\end{equation*}
If $\eta >
\max\{\delta^{-1}N^{-1/2|\Gamma|},\epsilon^{-1}N^{-1/2|\Lambda|}\}$
then this intersection has size greater than 1 and hence contains a
non-zero element say $d$. Moreover by the triangle inequality, for
any $k$ with $|k| \leq \min\{\delta N^{1/2|\Gamma|},\epsilon
N^{1/2|\Lambda|}\}$ we have $kd \in B(\Gamma,\eta\delta) \cap
B(\Lambda,\eta\epsilon)$.  Thus $(A''-A'')+(B''-B'')$ contains an arithmetic
progression of length
\begin{equation*}
L \geq \min\{\delta N^{1/2|\Gamma|},\epsilon N^{1/2|\Lambda|}\} \gg
K^{-O(1)}|A|^{1/CK^{3/4}\log K}
\end{equation*}
for some absolute $C>0$, and hence $(A-A)+(\alpha.A-\alpha.A)$ contains an arithmetic
progression $P$ of length $L$.

By Lemma \ref{lem.iteratedgrowth} with parameter $2l$ and some manipulations we get that
\begin{eqnarray*}
 |((A-A)+\alpha.(A-A)) + \alpha^2.((A-A)+\alpha.(A-A)) +\dots & & \\  + \alpha^{2(l-1)}.((A-A)+\alpha.(A-A))|& & \leq K^{O(l)}|A|
\end{eqnarray*}
However,  $(A-A)+\alpha.(A-A)$ contains $P$ as so we have (by the transcendance of $\alpha$) that the left hand side is at least $|P|^{l}$. Taking $l=\lceil 2CK^{3/4}\log K\rceil$ and inserting the lower bound for $|P|$, the result follows.
\end{proof}

\bibliographystyle{alpha}

\bibliography{master}

\end{document}